 \documentclass[12pt,reqno]{article}

 \usepackage{amsmath, amsthm, amsfonts}
 \usepackage{mathrsfs, amssymb}
 \usepackage{graphicx, color, bm}

 \numberwithin{equation}{section}

 \newtheorem{thm}{Theorem}[section]
 \newtheorem{cor}[thm]{Corollary}
 \newtheorem{lem}[thm]{Lemma}
 \newtheorem{prop}[thm]{Proposition}
 \theoremstyle{remark}
 
 \theoremstyle{example}

 \def\ar{\!\!\!&}\def\nnm{\nonumber}\def\ccr{\nnm\\}

 \def\mcr{\mathscr}\def\mbb{\mathbb}\def\mbf{\mathbf}

 \def\beqlb{\begin{eqnarray}}\def\eeqlb{\end{eqnarray}}
 \def\beqnn{\begin{eqnarray*}}\def\eeqnn{\end{eqnarray*}}
 \def\d{{\mbox{\rm d}}}

 \def\eqref#1{{\rm(\ref{#1})}}

\begin{document}

\
\bigskip\bigskip

\centerline{\Large\textbf{Moments of continuous-state branching}}

\smallskip

\centerline{\Large\textbf{processes with or without immigration}}

\bigskip

\centerline{Lina Ji and Zenghu Li}

\medskip

\centerline{School of Mathematical Sciences, Beijing Normal University,}

\centerline{Beijing 100875, People's Republic of China}

\centerline{E-mails: \tt lizh@bnu.edu.cn, jilina@mail.bnu.edu.cn}

\bigskip

{\narrower{\narrower

\centerline{\textbf{Abstract}}

\medskip

For a positive continuous function $f$ satisfying some standard conditions, we study the $f$-moments of continuous-state branching processes with or without immigration. The main results give criteria for the existence of the $f$-moments. The characterization of the processes in terms of stochastic equations given by Dawson and Li (2012) plays an essential role in the proofs.

\medskip

\noindent\textbf{Keywords and phrases:} branching process; continuous-state; immigration; moments; stochastic equation.

\par}\par}


\section{Introduction}

Branching processes in discrete state space were introduced as probabilistic models for the stochastic evolution of populations. For the basic theory of those processes we refer to Athreya and Ney (1972) and Harris (1965). Ji\v{r}ina (1958) defined continuous-state branching processes (CB-processes) in both discrete and continuous times. Those processes with continuous times were obtained in Lamperti (1967a) as weak limits of rescaled discrete branching processes. Lamperti (1967b) showed that they are in one-to-one correspondence with spectrally L\'{e}vy processes via simple random time changes. Continuous-state branching processes with immigration (CBI-processes) are more general population models taking into consideration the influence of the environments. They were introduced by Kawazu and Watanabe (1971) as rescaled limits of discrete branching processes with immigration; see also Aliev (1985). The approach of stochastic equations for CB- and CBI-processes have been developed by Dawson and Li (2006, 2012), Fu and Li (2010) and Li (2011) with some applications.

Moment properties play important roles in the study of limit theorems of branching processes. The integer-moments for the processes can be easily represented thanks to the simple forms of the generating functions or Laplace transforms of the distributions. The characterization of general function moments is usually more difficult. Suppose that $f$ is a positive continuous function on $[0,\infty)$ satisfying the following:

\smallskip

\noindent\textbf{Condition A.~} There exist constants $c\ge 0$ and $K> 0$ such that

(A1) $f$ is convex on $[c,\infty)$;

(A2) $f(xy)\le Kf(x)f(y)$ for all $x,y\in [c,\infty)$;

(A3) $f$ is bounded in $[0,c)$.

\smallskip

\noindent For a branching process with continuous time and discrete state space it was proved in Athreya (1969) that the existence of the $f$-moment is equivalent to that of its offspring distribution; see also Athreya and Ney (1972). The proof of Athreya (1969) was essentially based on a construction of the process from two sequences of random variables giving the split times and the progeny numbers. The result was generalized in Bingham (1976) to a CB-process for the function $f(x)= x^n$ with integer $n\ge 2$, which corresponds to integer-moments. A recursive formula for integer-moments of multi-type CBI-processes was given recently by Barczy et al.\ (2015). As far as we know, the result of Athreya (1969) has not been extended to the general $f$-moment in the continuous-state setting. The difficulty of such an extension lies in the fact that the CB-process cannot be constructed in the simple way as the discrete-sate process in Athreya (1969). We notice that a result on the $f$-moment of the CB-process for $f(x)= x\log x$ was presented in Section~5 of Grey (1974). It was mentioned there the topic would be studied elsewhere, but we could not find the subsequent work in the literature.

The purpose of this paper is to study general $f$-moments of CB- and CBI-processes with continuous time. Our two main theorems are stated in Section~2, giving criteria for the existence of the $f$-moments. The results yield immediately those of Bingham (1976) and Grey (1974). The proofs of the main theorems are given in Sections~3 and~4. Our strategy for the proofs is to use the characterization of the CB- and CBI-processes as strong solutions of stochastic equations established in Dawson and Li (2006, 2012). We shall need to give some slight generalizations of their results. Throughout the paper, we make the convention that, for $a\le b\in \mbb{R}$,
 \beqnn
\int_a^b = \int_{(a,b]}
 \quad\mbox{and}\quad
\int_a^\infty = \int_{(a,\infty)}.
 \eeqnn

\section{Main Results}

We first review some basic facts on CB- and CBI-processes with continuous time. The reader may refer to Kawazu and Watanabe (1971) for the details; see also Kyprianou (2014) and Li (2011). A \emph{branching mechanism} is a continuous function $\phi$ on $[0,\infty)$ with the representation
 \beqlb\label{2.1}
\phi(\lambda) = \beta\lambda + \frac{1}{2}\sigma^2 \lambda^2 + \int_0^\infty \big(e^{-z\lambda} - 1 + z\lambda 1_{\{z\le 1\}}\big)m(\d z),\qquad \lambda \ge 0,
 \eeqlb
where $\beta\in \mbb{R}$ and $\sigma\ge 0$ are constants, and $m(\d z)$ is a $\sigma$-finite measure on $(0,\infty)$ satisfying
 \beqnn
\int_0^\infty (1\wedge z^2) m(\d z) < \infty.
 \eeqnn
Throughout this paper, we assume
 \beqlb\label{2.2}
\int_{0+} \frac{1}{\phi(\lambda)} \d \lambda = \infty.
 \eeqlb
Then the \textit{CB-process} with branching mechanism $\phi$ is a conservative Markov process on $[0,\infty)$ with transition semigroup $(Q_t)_{t \ge 0}$ defined by
 \beqlb\label{2.3}
\int_{[0,\infty)} e^{-\lambda y}Q_t(x,\d y)
 =
\exp\{-xv_t(\lambda)\}, \qquad \lambda, x\ge 0,
 \eeqlb
where $t\rightarrow v_t(\lambda)$ is the unique positive solution of
 \beqlb\label{2.4}
v_t(\lambda) = \lambda - \int_0^t \phi(v_s(\lambda))\d s, \qquad \lambda,t\ge 0.
 \eeqlb

A generalization of the CB-process can be defined as follows. By an \textit{immigration mechanism} we mean a continuous positive function $\psi$ on $[0,\infty)$ given by
 \beqlb\label{2.5}
\psi(\lambda) = h\lambda + \int_0^\infty (1 - e^{- \lambda z})n(\d z),
 \eeqlb
where $h\ge 0$ is a constant and $n(\d z)$ is a $\sigma$-finite measure on $(0,\infty)$ satisfying
 \beqnn
\int_0^\infty (1\wedge z) n(\d z) < \infty.
 \eeqnn
It is well-known that there is an infinitely divisible probability measure $\gamma$ on $[0,\infty)$ so that $\psi=-\log L_\gamma$, where $L_\gamma$ is the Laplace transform of $\gamma$ defined by
 \beqnn
L_\gamma(\lambda) = \int_{[0,\infty)} e^{-\lambda z}\gamma(\d z), \qquad \lambda\ge 0,
 \eeqnn
A Markov process on $[0,\infty)$ is called \emph{CBI-process} with branching mechanism $\phi$ and immigration mechanism $\psi$ if it has transition semigroup $(Q^\gamma_t)_{t \ge 0}$ given by
 \beqlb\label{2.6}
\int_{[0,\infty)} e^{- \lambda y} Q^\gamma_t(x,\d y)
 =
\exp\bigg\{-x v_t(\lambda) - \int_0^t \psi(v_s(\lambda))\d s\bigg\},\qquad \lambda, x \ge 0.
 \eeqlb

The main results of this paper are the following:

\begin{thm}\label{t2.1}
Suppose that $f$ satisfies Condition A. Let $\{X_t: t\ge 0\}$ be CB-processes with $\mbf{P}(X_0> 0)> 0$. Then for any $t> 0$ we have $\mbf{P}f(X_t)< \infty$ if and only if $\mbf{P}f(X_0)< \infty$ and $\int_1^\infty f(z) m(\d z)< \infty$.
\end{thm}

\begin{thm}\label{t2.2}
Suppose that $f$ satisfies Condition A. Let $\{Y_t: t\ge 0\}$ be a CBI-process with $\mbf{P}(Y_0>0)> 0$. Then for every $t> 0$ we have $\mbf{P} f(Y_t)< \infty$ if and only if $\int_1^\infty f(z) (m+n)(\d z) < \infty$ and $\mbf{P} f(Y_0) < \infty.$
\end{thm}

For continuous-time branching processes and age dependent branching processes in discrete state space, some similar results as the above were established by Athreya (1969); see also Athreya and Ney (1972, p.153). By taking $f(x)= x^n$ or $f(x)= x\log x$ in Theorem~\ref{t2.1}, we obtain the results of Theorem~6.1 of Bingham (1976) and Section~5 of Grey (1974), respectively.

\section{Moments of CB-processes}

In this section, we discuss the $f$-moment of the CB-process with branching mechanism $\phi$ given by \eqref{2.1}. We shall first give a construction of the process in terms of a stochastic equation. This construction generalizes slightly the results of Dawson and Li (2006, 2012) and plays an important role in the study of the $f$-moment.

Let $(\Omega, \mathscr{G}, \mbf{P})$ a complete probability space with the augmented filtration $(\mathscr{G}_t)_{t\ge 0}$. Let $W(\d s,\d u)$ be a $(\mathscr{G}_t)$-time-space Gaussian white noise on $(0,\infty)^2$ based on the Lebesgue measure $\d s\d u$. Let $M(\d s,\d z,\d u)$ be a $(\mathscr{G}_t)$-time-space Poisson random measures on $(0,\infty)^3$ with intensity $\d sm(\d z)\d u$. Let $\tilde{M}(\d s,\d z,\d u)$ denote the compensated measure of $M(\d s,\d z,\d u)$. For any given $\mcr{G}_0$-measurable positive random variable $X_0$, we consider the stochastic integral equation
 \beqlb\label{3.1}
X_t \ar=\ar X_0 + \sigma\int_0^t \int_0^{X_{s-}}W(\d s,\d u) + \int_0^t \int_0^1 \int_0^{X_{s-}} z \tilde{M}(\d s,\d z,\d u) \cr
 \ar\ar\qquad
-\, \beta \int_0^t X_{s-} \d s + \int_0^t \int_1^\infty \int_0^{X_{s-}} z M(\d s,\d z,\d u).
 \eeqlb

\begin{thm}\label{t3.1} There is a unique positive strong solution to \eqref{3.1} and the solution $(X_t)_{t \ge 0}$ is a CB-process with transition semigroup $(Q_t)_{t\ge 0}$ defined by \eqref{2.3}. \end{thm}

\begin{proof} By applying Theorem~2.5 or Theorem~3.1 in Dawson and Li (2012) one can see the theorem holds if $(z\wedge z^2) m(\d z)$ is a finite measure on $(0,\infty)$; see also Dawson and Li (2006). Then for each integer $k\ge 1$ there is a unique positive strong solution $\{X_t^{(k)}: t\ge 0\}$ to the stochastic equation
 \beqlb\label{3.2}
X_t \ar=\ar X_0 + \sigma \int_0^t \int_0^{X_{s-}}W(\d s,\d u) + \int_0^t \int_0^1 \int_0^{X_{s-}} z \tilde{M}(\d s,\d z,\d u) \cr
 \ar\ar\qquad
-\, \beta \int_0^t X_{s-} \d s + \int_0^t \int_1^\infty \int_0^{X_{s-}} (z\wedge k) M(\d s,\d z,\d u).
 \eeqlb
In view of \eqref{3.2}, we have $X_t^{(k+1)}= X_t^{(k)}$ for $0\le t< S_k$ and $k\ge 1$, where $S_k= \inf\{t>0: X_t^{(k)} - X_{t-}^{(k)}\ge k\}$. It is easy to see that the process $t\mapsto X_t := \lim_{n \rightarrow \infty} X_t^{(k)}$ is a solution to \eqref{3.1}. The pathwise uniqueness of the solution of \eqref{3.1} follows from that of \eqref{3.2}. By Theorem~3.1 of Dawson and Li (2012) one sees that $\{X_t^{(k)}: t\ge 0\}$ is a CB-process with branching mechanism $\phi_k$ defined by
 \beqlb\label{3.3}
\phi_k(\lambda) = \beta\lambda + \frac{\sigma^2}{2}\lambda^2 + \int_0^\infty (e^{- \lambda (z\wedge k)} - 1 + \lambda z1_{\{z\le 1\}})m(\d z).
 \eeqlb
The transition semigroup $(Q_t^{(k)})_{t\ge 0}$ of this process is determined by
 \beqnn
\int_{[0,\infty)} e^{-\lambda y}Q_t^{(k)}(x,\d y)
 =
\exp\{-xv_t^{(k)}(\lambda)\},\qquad \lambda, x\ge 0,
 \eeqnn
where $t\mapsto v_t^{(k)}(\lambda)$ is the unique positive solution of
 \beqlb\label{3.4}
v_t^{(k)}(\lambda) = \lambda - \int_0^t \phi_k(v_s^{(k)}(\lambda))\d s,\qquad \lambda, t\ge 0.
 \eeqlb
By comparison theorem we see $v_t^{(k)}(\lambda)\le v_t^{(k+1)}(\lambda)\le v_t(\lambda)$, where $t\mapsto v_t(\lambda)$ is the unique positive solution to \eqref{2.4}. It follows that $v_t^{(k)}(\lambda)\to v_t(\lambda)$ increasingly as $k\rightarrow \infty$. Then $(X_t)_{t\ge 0}$ is a CB-process with branching mechanism $\phi$. \end{proof}

Let $\{X_t(x): t\ge 0\}$ be the solution of \eqref{3.1} with $X_0(x)=x\ge 0$. Then $\{X_t(x): t\ge 0\}$ is a CB-process with transition semigroup $(Q_t)_{t\ge 0}$.

\begin{thm}\label{t3.2} The path-valued process $x\mapsto \{X_t(x): t\ge 0\}$ has positive and independent increments. Furthermore, for any $y\ge x\ge 0$ the difference $\{X_t(y) - X_t(x): t\ge 0\}$ is a CB-process with initial value $y-x$. \end{thm}

\begin{proof} When $(z\wedge z^2) m(\d z)$ is a finite measure on $(0,\infty)$, the theorem is a consequence of Theorems~3.2 and~3.3 in Dawson and Li (2012). In the general case, it follows from the approximation of the solution given in the proof of Theorem~\ref{t3.1}. \end{proof}

We next study the existence of the $f$-moment of the CB-process. Instead of Condition~A, we here introduce the following more convenient condition:

\smallskip

\noindent\textbf{Condition B.~} There exists a constant $K > 0$ such that

(B1) $f(x)$ is convex and nondecreasing on $[0, \infty);$

(B2) $f(xy)\le Kf(x)f(y)$ for all $x, y\in [0, \infty);$

(B3) $f(x)> 1$ for all $x\in [0, \infty).$

\smallskip

\noindent This replacement of the condition is not essential. Indeed, as observed in Athreya and Ney (1972, p.154), for any unbounded function $f$ on $[0,\infty)$ satisfying Condition~A there is a constant $a\ge 0$ so that the function $x\mapsto f_a(x) := f(a\vee x)$ satisfies Condition~B. Of course, a probability measure on $[0,\infty)$ has finite $f$-moment if and only if it has finite $f_a$-moment.

Let $\tau_0(x)= 0$ and for $n\ge 1$ let $\tau_n(x)$ denote the $n$th jump time with jump size in $(1,\infty)$ of $\{X_t(x): t\ge 0\}$.

\begin{prop}\label{t3.3}
Suppose that $f$ satisfies Condition B. Then for any $t\ge 0$ and $y\ge x> 0$ we have
 \beqlb\label{3.5}
\mbf{P}[f(X_t(y))1_{\{t< \tau_n(y)\}}]\le Kf(1+y/x)\mbf{P}[f(X_t(x))1_{\{t< \tau_n(x)\}}].
 \eeqlb
\end{prop}

\begin{proof} Let $X_t^{(i)}(x) = X_t(ix) - X_t((i-1)x)$. By Theorem~\ref{t3.2}, $\{X_t^{(i)}(x): t\ge 0\}$, $i=1,2,\dots$ are i.i.d.\ CB-processes with $X_0^{(i)}(x) = x$. Let $\lfloor x \rfloor$ denote the largest integer smaller than or equal to $x\ge 0$. By Condition B we have
 \beqnn
\mbf{P}[f(X_t(y))1_{\{t< \tau_n(y)\}}]
 \ar=\ar
\mbf{P}\bigg[f\bigg(\sum_{i=1}^{\lfloor y/x\rfloor} X_t^{(i)}(x) + X_t(y)-X_t(\lfloor y/x\rfloor x)\bigg)1_{\{t< \tau_n(y)\}}\bigg] \cr
 \ar\le\ar
\mbf{P}\bigg[f\bigg(\sum_{i=1}^{\lfloor y/x\rfloor} X_t^{(i)}(x) + X_t(y) - X_t(y-x)\bigg) 1_{\{t< \tau_n(y)\}}\bigg] \cr
 \ar\le\ar
Kf(\lfloor y/x\rfloor+1)\mbf{P}\bigg\{f\bigg(\frac{1}{\lfloor y/x\rfloor+1} \bigg[\sum_{i=1}^{\lfloor y/x\rfloor} X_t^{(i)}(x) \cr
 \ar\ar\qquad\qquad\qquad
+\, X_t(y)-X_t(y-x)\bigg]\bigg)1_{\{t< \tau_n(y)\}}\bigg\} \cr
 \ar\le\ar
Kf(\lfloor y/x\rfloor+1)\mbf{P}\bigg\{\bigg(\frac{1}{\lfloor y/x\rfloor+1} \bigg[\sum_{i=1}^{\lfloor y/x\rfloor} f(X_t^{(i)}(x))1_{\{t< \tau_n^{(i)}(x)\}} \cr
 \ar\ar\qquad\qquad\qquad
+\, f(X_t(y)-X_t(y-x))1_{\{t< \sigma_n\}}\bigg]\bigg)\bigg\} \cr
 \ar\le\ar
Kf(y/x+1)\mbf{P}[f(X_t(x))1_{\{t< \tau_n(x)\}}],
 \eeqnn
where $\tau_n^{(i)}(x)$ and $\sigma_n$ denote the $n$th jump times of $\{X_t^{(i)}(x): t\ge 0\}$ and $\{X_t(y)-X_t(y-x): t\ge 0\}$ with jump size in $(1,\infty)$, respectively. That proves \eqref{3.5}. \end{proof}

\begin{cor}\label{t3.4}
Suppose that $f$ satisfies Condition B. Then for any $t\ge 0$ and $y\ge x> 0$ we have
 \beqlb\label{3.6}
\mbf{P}f(X_t(y))\le Kf(1+y/x)\mbf{P}f(X_t(x)).
 \eeqlb
Consequently, we have $\mbf{P}f(X_t(y))< \infty$ if and only if $\mbf{P}f(X_t(x))< \infty$.
\end{cor}

\begin{proof} By letting $n\to \infty$ in \eqref{3.5} we obtain the first result. The second one is then an immediate consequence. \end{proof}

\begin{cor}\label{t3.5}
Suppose that $f$ satisfies Condition B. Let $\{X_t: t\ge 0\}$ be a CB-process with branching mechanism $\phi$ and arbitrary initial distribution. Then we have
 \beqlb\label{3.7}
\mbf{P}f(X_t)\le \frac{1}{2}K^2f(2)\big[f(1)+\mbf{P}f(X_0)\big]\mbf{P}f(X_t(1)), \qquad t\ge 0.
 \eeqlb
\end{cor}

\begin{proof} Without loss of generality, we may assume $\{X_t: t\ge 0\}$ solves the stochastic equation \eqref{3.1}. By Theorem~\ref{t3.2} and the Markov property we have
 \beqnn
\mbf{P}[f(X_t)|\mcr{G}_0]\le Kf(1+X_0)\mbf{P}f(X_t(1))
 \le
\frac{1}{2}K^2f(2)\big[f(1) + f(X_0)\big]\mbf{P}f(X_t(1)).
 \eeqnn
Then we get \eqref{3.7} by taking the expectation. \end{proof}

\begin{prop}\label{t3.6}
Suppose that $f$ satisfies Condition B and $\mbf{P}f(X_t(x))< \infty$ for some $x>0$ and $t\ge 0$. Let $\{X_t: t\ge 0\}$ be a CB-process with branching mechanism $\phi$ and arbitrary initial distribution. Then $\mbf{P}f(X_t)< \infty$ if and only if $\mbf{P}f(X_0)< \infty$.
\end{prop}

\begin{proof} Without loss of generality, we may assume $\{X_t: t\ge 0\}$ solves the stochastic equation \eqref{3.1}. Suppose that $\mbf{P}f(X_0)< \infty$. By Corollaries~\ref{t3.4} and~\ref{t3.5} we have $\mbf{P}f(X_t)< \infty$. Conversely, suppose that $\mbf{P}f(X_t)< \infty$. As in the proof of Proposition~\ref{3.3}, let $\lfloor x \rfloor$ denote the largest integer smaller than or equal to $x\ge 0$. By Condition B we have
 \beqnn
\mbf{P}f(X_0)\le \mbf{P}f(\lfloor X_0 \rfloor+1)
\le
\frac{1}{2}Kf(2)\{\mbf{P}f(\lfloor X_0\rfloor)+f(1)\}.
 \eeqnn
Then it suffices to show $\mbf{P}f(\lfloor X_0\rfloor)< \infty$. From Proposition~3.1 in Li (2011) and the proof of Theorem~\ref{t3.1}, we see $v_t(\lambda)> 0$ for any $\lambda>0$. By \eqref{2.3} it follows that $\mbf{P}(X_t(1)\in (0,\infty)) = Q_t(1,(0,\infty))> 0$. Then the infinite divisibility of $Q_t(1,\cdot)$ implies the existence of $\epsilon>0$ so that $\mbf{P}(X_t^{(i)}\ge \epsilon) = \mbf{P}(X_t(1)\ge \epsilon)\in (0,1)$. Now define the sequence of i.i.d.\ random variables $\{\delta_1,\delta_2,\dots\}$ by
 \beqnn
\delta_i =
\left\{
\begin{array}{ll}
 1, \ar \hbox{if $X_t^{(i)}\ge \epsilon$;} \cr
 0, \ar \hbox{otherwise.}
\end{array}
\right.
 \eeqnn
Then $\mbf{P}(\delta_i=1) = \mbf{P}(X_t^{(i)}\ge \epsilon)\in (0,1)$. Observe that
 \beqnn
\sum_{i=1}^{\lfloor X_0\rfloor} \delta_i
 \le
\epsilon^{-1}\sum_{i=1}^{\lfloor X_0\rfloor} X_t^{(i)}
 \le
\epsilon^{-1}X_t.
 \eeqnn
By Condition B we have
 \beqnn
\mbf{P}f\bigg(\sum_{i=1}^{\lfloor X_0\rfloor} \delta_i\bigg)
 \le
\mbf{P}f(\epsilon^{-1}X_t)\le Kf(\epsilon^{-1})\mbf{P}f(X_t)< \infty.
 \eeqnn
By the property of independent increments of the noises in \eqref{3.1}, the $\mcr{G}_0$-measurable random variable $X_0$ is independent of $\{X_t^{(i)}: t\ge 0\}$, $i=1,2,\dots$. Then $\lfloor X_0\rfloor$ is independent of the sequence $\{\delta_1,\delta_2,\dots\}$. By Lemmas~4 and~5 of Athreya and Ney (1972, pp.156--157) we have $\mbf{P}f(\lfloor X_0\rfloor)< \infty$.
\end{proof}

\begin{lem}\label{t3.7}
Suppose that $f$ satisfies Condition B and $\int_1^\infty z^n m(\d z)< \infty$ for every $n\ge 1$. Then for any $x>0$ the function $t\mapsto \mbf{P}f(X_t(x))$ is locally bounded on $[0,\infty)$.
\end{lem}

\begin{proof}
It is easy to see that the function $z\mapsto g(z) := f(e^z)$ is convex and nondecreasing on $[0,\infty)$. By Condition B, there exists a constant $K > 0,$ such that
 \beqnn
g(z+y) = f(e^ze^y)\le Kf(e^z)f(e^y) = Kg(z)g(y), \qquad z,y\ge 0.
 \eeqnn
By Lemma~25.5 of Sato (1999, p.160), there is some $c> 0$ and some integer $n\ge 1$ so that $g(z)\le ce^{nz}$ for $z\ge 0$. It follows that $f(z)\le cz^n$ for $z\ge 1$. By Theorem~6.1 of Bingham (1976) or Theorem~4.3 of Barczy et al.\ (2015), we can get $\mbf{P}(X_t(x)^n)< \infty$. Then
 \beqnn
\mbf{P}f(X_t(x))\le f(1) + c\mbf{P}[X_t(x)^n]< \infty.
 \eeqnn
Since $\int_1^\infty z m(\d z)< \infty$, we can rewrite \eqref{2.1} into
 \beqlb\label{3.8}
\phi(\lambda) = b\lambda + \frac{1}{2}\sigma^2 \lambda^2 + \int_0^\infty \big(e^{-\lambda z} - 1 + \lambda z\big)m(\d z),\qquad \lambda \ge 0,
 \eeqlb
where
 \beqnn
b = \beta-\int_1^\infty z m(\d z).
 \eeqnn
In this case, we have
 \beqnn
\int_{[0,\infty)} y Q_t(x,\d y)
 =
xe^{-xbt},\qquad t,x\ge 0.
 \eeqnn
See Li (2011, Chapter 3). Then the Markov property implies that $t\mapsto W_t(x) := e^{bt}X_t(x)$ is a martingale, and hence $t\mapsto f(W_t(x))$ is a positive sub-martingale. For $t\in [0,T]$ we have
 \beqnn
\mbf{P}f(X_t(x)) \ar=\ar \mbf{P}f(e^{-bt}W_t(x))
 \le
Kf(e^{-bt})\mbf{P}f(W_t(x)) \ccr
 \ar\le\ar
Kf(e^{-bt})\mbf{P}f(W_T(x))
 \le
Kf(1\vee e^{-bT})\mbf{P}f(e^{bT}X_T(x)) \ccr
 \ar\le\ar
K^2f(1\vee e^{-bT})f(e^{bT})\mbf{P}f(X_T(x)).
 \eeqnn
Then $t\mapsto \mbf{P}f(X_t(x))$ is a locally bounded function. \end{proof}

Recall that $\tau_n(x)$ is the $n$th jump time with jump size in $(1,\infty)$ of the process $\{X_t(x): t\ge 0\}$. Let $G_x(\d t) = \mbf{P}(\tau_1(x)\in \d t)$ and $\mu_n(t) = \mbf{P}(f(X_t(1)); t< \tau_n(1))$ for $t\ge 0$. A characterization of the distribution $G_x(\d t)$ can be derived from Theorem~3.2 of He and Li (2016).

\begin{prop}\label{t3.8}
Suppose that $f$ satisfies Condition B and $\int_1^\infty f(z) m(\d z)< \infty$. Then for every $T > 0$ there are constants $c_1(T)\ge 0$ and $c_2(T)\ge 0$ so that
 \beqlb\label{3.9}
\mu_n(t)\le c_1(T) + c_2(T) \int_0^t \mu_{n-1}(t-u)G_1(\d u), \qquad t \in (0,T], ~ n\ge 0.
 \eeqlb
\end{prop}

\begin{proof} To avoid triviality, we assume $m(1,\infty)> 0$. Recall that $\{X_t(x): t\ge 0\}$ is the strong solution of \eqref{2.1} with $X_0(x)=x\ge 0$. On the same probability space, let $\{Z_t(x): t\ge 0\}$ be the strong solution of the stochastic equation
 \beqlb\label{3.10}
Z_t(x) \ar=\ar x - \beta \int_0^t Z_{s-}(x) \d s + \sigma\int_0^t \int_0^{Z_{s-}(x)}W(\d s,\d u) \cr
 \ar\ar\qquad
+ \int_0^t \int_0^1 \int_0^{Z_{s-}(x)} z \tilde{M}(\d s,\d z,\d u).
 \eeqlb
Then $\{Z_t(x): t\ge 0\}$ is a CB-process with branching mechanism
 \beqnn
\phi_1(\lambda) = \beta\lambda + \frac{1}{2}\sigma^2 \lambda^2 + \int_0^1 (e^{-\lambda z}-1+\lambda z)m(\d z),\qquad \lambda \ge 0.
 \eeqnn
Let $W$ denote the space of all c\`{a}dl\`{a}g paths $t\mapsto x(t)$ from $[0,\infty)$ to itself equipped with the Skorokhod topology. Let $\mcr{F} = \sigma\{x(s): s\ge 0\}$ and $\mcr{F}_t = \sigma\{x(s): 0\le s\le t\}$, $t\ge 0$ be the natural $\sigma$-algebras on $W$. Let $\mbf{P}_x$ denote the distribution of $\{X_t(x): t\ge 0\}$ on $W$. Then $(W,\mcr{F},\mcr{F}_t,\mbf{P}_x)$ is the canonical realization of the CB-process with transition semigroup $(Q_t)_{t\ge 0}$. Let $\sigma_n$ denote the $n$th jump time of $\{x(t): t\ge 0\}$ with jump size in $(1,\infty)$. In view of \eqref{3.1} and \eqref{3.10}, we can use the notation in the theory of Markov processes to write
 \beqnn
\mu_n(t) \ar=\ar \mbf{P}[f(X_t(1))1_{\{t< \tau_1(1)\}}] + \mbf{P}[f(X_t(1)) 1_{\{\tau_1(1)\le t< \tau_n(1)\}}] \ccr
 \ar=\ar
\mbf{P}[f(Z_t(1))1_{\{t< \tau_1(1)\}}] + \mbf{P}\{1_{\{\tau_1(1)\le t\}} \mbf{P}[f(X_t(1)) 1_{\{t< \tau_n(1)\}}| \mcr{G}_{\tau_1(1)}]\} \ccr
 \ar\le\ar
\mbf{P}f(Z_t(1)) + \mbf{P}\{1_{\{\tau_1(1)\le t\}} \mbf{P}_{X_{\tau_1(1)}(1)} [f(x(t-\tau_1(1))) 1_{\{t-\tau_1(1)< \sigma_{n-1}\}}]\} \ccr
 \ar=\ar
\mbf{P}f(Z_t(1)) + \mbf{P}\{1_{\{\tau_1(1)\le t\}} \mbf{P}_{Z_{\tau_1(1)}(1) + \Delta X_{\tau_1(1)}(1)}[f(x(t-\tau_1(1))) 1_{\{t-\tau_1(1)< \sigma_{n-1}\}}]\}.
 \eeqnn
From the stochastic equation \eqref{3.1} we see $\mbf{P}(\tau_1(1)\in \d s,\Delta X_{\tau_1(1)}(1)\in \d z) = G_1(\d s)\hat{m}_1(\d z)$, where $\hat{m}_1(\d z) = m(1,\infty)^{-1}1_{\{z>1\}} m(\d z)$. Then, by Corollary~\ref{t3.4},
 \beqnn
\mu_n(t) \ar\le\ar \mbf{P}f(Z_t(1)) + \int_0^tG_1(\d s) \int_1^\infty \mbf{P}\{\mbf{P}_{Z_s(1)+z} [f(x(t-s)) 1_{\{t-s< \sigma_{n-1}\}}]\}\hat{m}_1(\d z) \cr
 \ar\le\ar
c_1(T) + K\int_0^t \mu_{n-1}(t-s) G_1(\d s) \int_1^\infty \mbf{P}f(Z_s(1)+z+1) \hat{m}_1(\d z),
 \eeqnn
where $c_1(T)= \sup_{0\le t\le T}\mbf{P}f(Z_t(1))$ by Lemma~\ref{t3.7} and
 \beqnn
\ar\ar\int_1^\infty \mbf{P}f(Z_u(1)+z+1)\hat{m}_1(\d z) \cr
 \ar\ar\qquad\qquad
\le Kf(3)\int_1^\infty\mbf{P}f\Big(\frac{1}{3}\{Z_u(1)+z+1\}\Big)\hat{m}_1(\d z) \cr
 \ar\ar\qquad\qquad
\le \frac{1}{3}Kf(3)\bigg[\mbf{P}f(Z_u(1)) + \int_1^\infty f(z)\hat{m}_1(\d z) + f(1)\bigg] \cr
 \ar\ar\qquad\qquad
\le \frac{1}{3}Kf(3)\bigg[c_1(T) + \int_1^\infty f(z)\hat{m}_1(\d z) + f(1)\bigg] =: c(T).
 \eeqnn
Then we get \eqref{3.9} with $c_2(T) = Kc(T)$.
\end{proof}

\begin{prop}\label{t3.9}
Suppose that $f$ satisfies Condition B and $\int_1^\infty f(z) m(\d z)< \infty$. Then for any $x\ge 0$ the function $t\mapsto \mbf{P}f(X_t(x))$ is locally bounded on $[0,\infty)$.
\end{prop}

\begin{proof} Let $c_1(T)\ge 0$ and $c_2(T)\ge 0$ be provided by Proposition~\ref{t3.8}. By Lemma~2 of Athreya and Ney (1972, p.145) there is a bounded positive function $t\mapsto \mu(t)$ on $[0,T]$ satisfying
 \beqlb\label{3.11}
\mu(t) = c_1(T) + c_2(T) \int_0^t \mu(t-u)\d G_1(u), \qquad 0\le t\le T.
 \eeqlb
In view of \eqref{3.9} and \eqref{3.11}, one can show by induction that $\mu_n(t)\le \mu(t)$ for all $0\le t\le T$ and $n\ge 1$. Since $\sigma_n\to \infty$ as $n\to \infty$ we have $\mbf{P}f(X_t(1)) = \lim_{n\to \infty}\mu_n(t)\le \mu(t)$. Then the result follows by Corollary~\ref{t3.4}. \end{proof}

\begin{proof}[Proof of Theorem~\ref{t2.1}] Without loss of generality, we may assume $\{X_t: t\ge 0\}$ solves the stochastic equation \eqref{3.1}. Suppose that $\mbf{P}f(X_0)< \infty$ and $\int_1^\infty f(z) m(\d z)< \infty$. Then $\mbf{P}f(X_t(1))< \infty$ by Proposition~\ref{t3.9} and $\mbf{P}f(X_t)< \infty$ by Corollary~\ref{t3.5}. Conversely, suppose that $\mbf{P}f(X_t)< \infty$ for some $t> 0$. Let $\tau_n$ denote the $n$th jump time of $\{X_t: t\ge 0\}$ with jump size in $(1,\infty)$ and let $G(\d t) = \mbf{P}(\tau_1\in\d t)$. Using the notation introduced in the proof of Proposition~\ref{t3.8}, we have
 \beqnn
\mbf{P}f(X_t) \ar\ge\ar \mbf{P}[f(X_t)1_{\{\tau_1\le t\}}]
 =
\mbf{P}\{1_{\{\tau_1\le t\}} \mbf{P}[f(X_t)| \mcr{G}_{\tau_1}]\} \ccr
 \ar=\ar
\mbf{P}\{1_{\{\tau_1\le t\}} \mbf{P}_{X_{\tau_1}} f(x(t-\tau_1))\}
 \ge
\mbf{P}\{1_{\{\tau_1\le t\}} \mbf{P}_{\Delta X_{\tau_1}}f(x(t-\tau_1))\} \cr
 \ar=\ar
\int_0^tG(\d s) \int_1^\infty \mbf{P}_zf(x(t-s)) \hat{m}_1(\d z) \cr
 \ar=\ar
\int_0^tG(\d s) \int_1^\infty \mbf{P}f(X_{t-s}(z)) \hat{m}_1(\d z) \cr
 \ar\ge\ar
\int_0^tG(\d s) \int_1^\infty \mbf{P} f\bigg(\sum_{i=1}^{\lfloor z\rfloor} X_{t-s}^{(i)}\bigg) \hat{m}_1(\d z).
 \eeqnn
By Theorem~3.5 of Li (2011, p.59) we have $\mbf{P}(X_t(1)> 0)> 0$. To avoid triviality, we assume $m(1,\infty)>0$, so \eqref{3.1} implies that $t\mapsto G(0,t]$ is strictly increasing on $[0,\infty)$. Then there must be some $s\in (0,t]$ so that
 \beqnn
\int_1^\infty \mbf{P} f\bigg(\sum_{i=1}^{\lfloor z\rfloor}X_{t-s}^{(i)}\bigg) \hat{m}_1(\d z) < \infty.
 \eeqnn
By Lemmas~4 and~5 of Athreya and Ney (1972, pp.156--157) we have
 \beqnn
\int_1^\infty f(\lfloor z\rfloor) \hat{m}_1(\d z) < \infty.
 \eeqnn
It follows that
 \beqnn
\int_1^\infty f(z) \hat{m}_1(\d z)
 \ar\le\ar
\int_1^\infty f(\lfloor z\rfloor+1) \hat{m}_1(\d z) \cr
 \ar\le\ar
Kf(2)\int_1^\infty f\Big(\frac{1}{2}\big[\lfloor z\rfloor+1\big]\Big) \hat{m}_1(\d z) \cr
 \ar\le\ar
\frac{1}{2}Kf(2)\int_1^\infty \big[f(\lfloor z\rfloor) + f(1)\big] \hat{m}_1(\d z) \cr
 \ar =\ar
\frac{1}{2}Kf(2)\bigg[\int_1^\infty f(\lfloor z\rfloor) \hat{m}_1(\d z) + f(1)\bigg]< \infty,
 \eeqnn
which implies $\int_1^\infty f(z) m(\d z)< \infty$. Then we have $\mbf{P}f(X_t(1))< \infty$ by Proposition~\ref{t3.9} and $\mbf{P}f(X_0)< \infty$ by Proposition~\ref{t3.6}. \end{proof}

\section{Moments of CBI-processes}

In this section, we discuss the $f$-moment of the CBI-process. As in the last section, we first give a construction of the process in terms of a stochastic equation.

Let $(\Omega, \mathscr{G}, \mbf{P})$ a complete probability space with the augmented filtration $(\mathscr{G}_t)_{t\ge 0}$. Let $W(\d s,\d u)$ be a $(\mathscr{G}_t)$-time-space Gaussian white noise on $(0,\infty)^2$ based on the Lebesgue measure $\d s \d u$. Let $M(\d s,\d z,\d u)$ and $N(\d s,\d z)$ be $(\mathscr{G}_t)$-time-space Poisson random measures on $(0,\infty)^3$ and $(0,\infty)^2$ with intensities $\d sm(\d z) \d u$ and $\d sn(\d z)$, respectively. Suppose that $W(\d s,\d u)$, $M(\d s,\d z,\d u)$ and $N(\d s,\d z)$ are independent of each other. Let $\tilde{M}(\d s,\d z,\d u)$ denote the compensated measure of $M(\d s,\d z,\d u)$. For any given $\mcr{G}_0$-measurable positive random variable $Y_0$, we consider the stochastic integral equation
 \beqlb\label{4.1}
Y_t \ar=\ar Y_0 + \sigma\int_0^t \int_0^{Y_{s-}}W(\d s,\d u) + \int_0^t \int_0^1 \int_0^{Y_{s-}} z \tilde{M}(\d s,\d z,\d u) \cr
 \ar\ar\qquad
+ \int_0^t (h - \beta Y_s) \d s + \int_0^t \int_1^\infty \int_0^{Y_{s-}} z M(\d s,\d z,\d u) \cr
 \ar\ar\qquad
+ \int_0^t \int_0^\infty z N(\d s,\d z).
 \eeqlb

\begin{thm}\label{t4.1}
There is a unique positive strong solution to \eqref{4.1} and the solution $(Y_t)_{t \ge 0}$ is a CBI-process with transition semigroup $(Q^\gamma_t)_{t\ge 0}$ defined by \eqref{2.6}.
\end{thm}

\begin{thm}\label{t4.2}
For any $x\ge 0$ let $\{Y_t(x): t\ge 0\}$ be the solution to \eqref{4.1} with $Y_0(x)= x\ge 0$. Then the path-valued process $x\mapsto \{Y_t(x): t\ge 0\}$ has positive and independent increments. Furthermore, for any $y\ge x\ge 0$ the difference $\{Y_t(y) - Y_t(x): t\ge 0\}$ is a CB-process with initial value $y-x$. \end{thm}

The above theorems generalize the results of Dawson and Li (2012). We here omit their proofs since the arguments are quite similar to those for the corresponding results in Section~3.

\begin{prop}\label{t4.3}
Suppose that $f$ satisfies Condition B. Let $\{X_t: t\ge 0\}$ be a CB-process and $\{Y_t: t\ge 0\}$ a CBI-process with $X_0 \overset{d}= Y_0$. Then
 \beqlb\label{4.2}
\mbf{P}f(Y_t)\le \frac{1}{2}Kf(2)\big[\mbf{P}f(Y_t(0)) + \mbf{P}f(X_t)\big], \qquad t\ge 0.
 \eeqlb
\end{prop}

\begin{proof} Without loss of generality, we assume $\{Y_t: t\ge 0\}$ and $\{X_t: t\ge 0\}$ are solutions of \eqref{4.1} and \eqref{3.1}, respectively, with $Y_0 = X_0$. Since $f$ satisfies Condition B, we have
 \beqnn
\mbf{P}f(Y_t) \ar=\ar \mbf{P}f(Y_t(0) + Y_t-Y_t(0)) \cr
\ar\le\ar
Kf(2)\mbf{P}f\Big(\frac{1}{2}[Y_t(0) + Y_t-Y_t(0)]\Big) \cr
 \ar\le\ar
\frac{1}{2}Kf(2)\big[\mbf{P}f(Y_t(0)) + \mbf{P}f(Y_t-Y_t(0))\big] \cr
 \ar=\ar
\frac{1}{2}Kf(2)\big[\mbf{P}f(Y_t(0)) + \mbf{P}f(X_t)\big],
 \eeqnn
where the last equality follows by Theorem~\ref{t4.2}. \end{proof}

\begin{lem}\label{t4.4}
Suppose that $f$ satisfies Condition B and $\int_1^\infty z^n (m + n)(\d z) < \infty$ for every $n \geq 1.$ Then for any $x \geq 0$ the function $t \rightarrow \mbf{P} f(Y_t(x))$ is locally bounded on $[0,\ \infty).$
\end{lem}

\begin{proof} The follows in the same way as in the proof of Lemma~\ref{t3.7} as one notices the process $t\to e^{bt}Y_t(x)$ is a sub-martingale. We leave the details to the reader. \end{proof}

Let $\zeta_0(x)= 0$ and let $\zeta_n(x)$ be the $n$th jump time of $\{Y_t(x): t\ge 0\}$ with jump size in $(1,\infty)$. Let $H(\d t) = \mbf{P}(\zeta_1(0)\in\d t)$ and $\nu_n(t) = \mbf{P}(f(Y_t(0)); t< \zeta_n(0))$ for $t\ge 0$. A characterization of the distribution $H(\d t)$ was given by He and Li (2016).

\begin{prop}\label{t4.5}
Suppose that $f$ satisfies Condition B and $\int_1^\infty f(z) (m+n)(\d z)< \infty$. Then for every $T> 0$ there is a constant $0\le c_3(T)< \infty$ so that
 \beqlb\label{4.3}
\nu_n(t)\le c_3(T) + \frac{1}{2}Kf(2)\int_0^t \nu_{n-1}(t-s) H(\d s), \qquad 0\le t\le T, ~ n\ge 1.
 \eeqlb
\end{prop}

\begin{proof} Let $(W, \mcr{F}, \mcr{F}_t, x(t))$ be as in the proof of Proposition~\ref{t3.8}. Let $\mbf{P}_x$ and $\mbf{P}^\gamma_x$ denote the laws on $(W, \mcr{F})$ of $\{X_t(x): t\ge 0\}$ and $\{Y_t(x): t\ge 0\}$, respectively. Then $(W, \mcr{F}, \mcr{F}_t, x(t), \mbf{P}_x)$ is a canonical realization of the CB-process and $(W, \mcr{F}, \mcr{F}_t, x(t), \mbf{P}^\gamma_x)$ is a canonical realization of the CBI-process. Let us also consider the stochastic equation
 \beqlb\label{4.4}
Z_t \ar=\ar Z_0 + \sigma\int_0^t \int_0^{Z_{s-}}W(\d s,\d u) + \int_0^t \int_0^1 \int_0^{Z_{s-}} z \tilde{M}(\d s,\d z,\d u) \cr
 \ar\ar\qquad
+ \int_0^t (h-\beta Z_{s-}) \d s + \int_0^t \int_0^1 z N(\d s,\d z).
 \eeqlb
Let $\{Z_t(x): t\ge 0\}$ denote the solution with $Z_0(x) = x\ge 0$. In view of \eqref{4.1} and \eqref{4.4}, we have
 \beqnn
\nu_n(t) \ar=\ar \mbf{P}[f(Y_t(0))1_{\{t< \zeta_1(0)\}}] + \mbf{P}[f(Y_t(0)) 1_{\{\zeta_1(0)\le t< \zeta_n(0)\}}] \cr
 \ar=\ar
\mbf{P}[f(Z_t(0))1_{\{t< \zeta_1(0)\}}] + \mbf{P}\{1_{\{\zeta_1(0)\le t\}} \mbf{P}[f(Y_t(0)) 1_{\{t< \zeta_n(0)\}}| \mcr{G}_{\zeta_1(0)}]\} \cr
 \ar\le\ar
\mbf{P}f(Z_t(0)) + \mbf{P}\{1_{\{\zeta_1(0)\le t\}} \mbf{P}^\gamma_{Y_{\zeta_1(0)}(0)} [f(x(t-\zeta_1(0))) 1_{\{t-\zeta_1(0)< \sigma_{n-1}\}}]\} \cr
 \ar=\ar
\mbf{P}f(Z_t(0)) + \mbf{P}\{1_{\{\zeta_1(0)\le t\}} \mbf{P}^\gamma_{Z_{\zeta_1(0)}(0) + \Delta Y_{\zeta_1(0)}(0)}[f(x(t-\zeta_1(0))) 1_{\{t-\zeta_1(0)< \sigma_{n-1}\}}]\} \cr
 \ar\le\ar
c_0(T) + \mbf{P}\bigg\{\int_0^t H(\d s) \int_1^\infty \mbf{P}^\gamma_{Z_s(0)+z} [f(x(t-s)) 1_{\{t-s< \sigma_{n-1}\}}]\eta_s(\d z)\bigg\},
 \eeqnn
where $c_0(T) = \sup_{0\le t\le T}\mbf{P}f(Z_t(0))$ by Lemma~\ref{t4.4} and
 \beqnn
\eta_s(\d z)
 =
1_{\{Y_{s-}(0)m(1,\infty) + n(1,\infty)> 0\}}\frac{Y_{s-}(0)m(\d z)+n(\d z)} {Y_{s-}(0)m(1,\infty)+n(1,\infty)}.
 \eeqnn
Observe that $\eta_s(\d z)\le (\hat{m}_1+\hat{n}_1)(\d z)$. By Theorem~\ref{t4.2} and Corollary~\ref{t3.4},
 \beqnn
\nu_n(t) \ar\le\ar c_0(T) + \frac{1}{2}Kf(2)\mbf{P}\bigg\{\int_0^t H(\d s) \int_1^\infty \mbf{P}_{Z_s(0)+z} [f(x(t-s)) 1_{\{t-s< \sigma_{n-1}\}}]\eta_s(\d z)\bigg\} \cr
 \ar\ar\qquad
+\, \frac{1}{2}Kf(2)\mbf{P}\bigg\{\int_0^t H(\d s)\int_1^\infty\mbf{P}^\gamma_0 [f(x(t-s)) 1_{\{t-s< \sigma_{n-1}\}}]\eta_s(\d z)\bigg\} \cr
 \ar\le\ar
c_0(T) + \frac{1}{2}K^2f(2)\int_0^t \mu_{n-1}(t-s) H(\d s) \int_1^\infty \mbf{P}f(Z_s(0)+z+1)\eta_s(\d z) \cr
 \ar\ar\qquad
+\,\frac{1}{2}Kf(2)\int_0^t \mbf{P}^\gamma_0 [f(x(t-s)) 1_{\{t-s< \sigma_{n-1}\}}] H(\d s) \cr
 \ar\le\ar
c_0(T) + \int_0^t \mu(t-s) h_0(s)H(\d s) + \frac{1}{2}Kf(2) \int_0^t \nu_{n-1}(t-s) H(\d s),
 \eeqnn
where
 \beqnn
h_0(s) \ar=\ar \frac{1}{2}K^2f(2)\int_1^\infty \mbf{P}f(Z_s(0)+z+1) (\hat{m}_1+\hat{n}_1)(\d z) \cr
 \ar\le\ar
\frac{1}{6}K^3f(2)f(3)\bigg\{2c_0(T) + \int_1^\infty f(z)(\hat{m}_1+\hat{n}_1)(\d z) + 2f(1)\bigg\} =: c_4(T).
 \eeqnn
It is easy to see that
 \beqnn
c_3(T) := c_0(T) + c_4(T)\sup_{0\le t\le T}\int_0^t \mu(t-s)H(\d s)< \infty.
 \eeqnn
Then we have \eqref{4.3}. \end{proof}

\begin{proof}[Proof of Theorem~\ref{t2.2}] Suppose that $\mbf{P}f(Y_0)< \infty$ and $\int_1^\infty f(z) (m+n)(\d z)< \infty$. Using Proposition~\ref{t4.5} we see as in the proof of Proposition~\ref{t3.9} that $\mbf{P}f(Y_t(0))< \infty$. Then $\mbf{P}f(Y_t)< \infty$ by Theorem~\ref{t2.1} and Proposition~\ref{t4.3}. Conversely, suppose that $\mbf{P}f(Y_t)< \infty$ for some $t> 0$. Let $\{X_t: t\ge 0\}$ are solution of \eqref{3.1} with $Y_0 = X_0$. By Theorem~\ref{t4.2} we see
 \beqnn
\mbf{P}f(X_t)=\mbf{P}f(Y_t-Y_t(0))\le \mbf{P}f(Y_t)< \infty.
 \eeqnn
Then Theorem~\ref{t2.1} implies $\mbf{P}f(Y_0) = \mbf{P}f(X_0)< \infty$. Moreover, using the notation introduced in the proof of Proposition~\ref{t4.5}, we have
 \beqnn
\mbf{P}f(Y_t) \ar\ge\ar \mbf{P}[f(Y_t)1_{\{\zeta_1\le t\}}]
 =
\mbf{P}\{1_{\{\zeta_1\le t\}} \mbf{P}[f(Y_t)| \mcr{G}_{\zeta_1}]\} \ccr
 \ar=\ar
\mbf{P}\{1_{\{\zeta_1\le t\}} \mbf{P}^\gamma_{Y_{\zeta_1}} f(x(t-\zeta_1))\}
 \ge
\mbf{P}\{1_{\{\zeta_1\le t\}} \mbf{P}_{\Delta Y_{\zeta_1}}f(x(t-\zeta_1))\} \cr
 \ar=\ar
\int_0^t \mbf{P}\bigg\{\int_1^\infty \mbf{P}_zf(x(t-s)) \eta_s(\d z)\bigg\} H(\d s).
 \eeqnn
To avoid triviality, in the following we assume $(m+n)(1,\infty)>0$. From \eqref{4.1} we see $t\mapsto H(0,t]$ is strictly increasing on $[0,\infty)$. Then there must be some $s\in (0,t]$ so that, a.s.,
 \beqnn
\int_1^\infty \mbf{P}_zf(x(t-s)) \eta_s(\d z)< \infty,
 \eeqnn
where
 \beqnn
\eta_s(\d z)
 =
1_{\{Y_{s-}m(1,\infty)+n(1,\infty)> 0\}}\frac{Y_{s-}m(\d z) + n(\d z)} {Y_{s-}m(1,\infty)+n(1,\infty)}.
 \eeqnn
Since $\{Y_t: t\ge 0\}$ is a Hunt process, we have $\mbf{P}(Y_{s-}=Y_s) = 1$. Let $\{X_t: t\ge 0\}$ be the solution of \eqref{3.1} with $X_0=Y_0$. By comparison we have a.s. $Y_s\ge X_s$. Then Theorem~3.5 of Li (2011, p.59) implies that $\mbf{P}(Y_{s-}> 0) = \mbf{P}(Y_s> 0)\ge \mbf{P}(X_s> 0)> 0$. It follows that
 \beqnn
\int_1^\infty \mbf{P}_zf(x(t-s)) (\hat{m}_1+\hat{n}_1)(\d z)< \infty,
 \eeqnn
and hence
 \beqnn
\ar\ar\int_1^\infty \mbf{P} f\bigg(\sum_{i=1}^{\lfloor z\rfloor}X_{t-s}^{(i)}\bigg) (\hat{m}_1+\hat{n}_1)(\d z) \cr
 \ar\ar\qquad
= \int_1^\infty \mbf{P}_{\lfloor z\rfloor}f(x(t-s)) (\hat{m}_1+\hat{n}_1)(\d z)< \infty.
 \eeqnn
By Lemmas~4 and~5 of Athreya and Ney (1972, pp.156--157) we have
 \beqnn
\int_1^\infty f(\lfloor z\rfloor) (\hat{m}_1 + \hat{n}_1)(\d z) < \infty.
 \eeqnn
It follows that
 \beqnn
\int_1^\infty f(z) (\hat{m}_1 + \hat{n}_1)(\d z)
 \ar\le\ar
\int_1^\infty f(\lfloor z\rfloor+1) (\hat{m}_1 + \hat{n}_1)(\d z) \cr
 \ar\le\ar
Kf(2)\int_1^\infty f\Big(\frac{1}{2}\{\lfloor z\rfloor+1\}\Big) (\hat{m}_1 + \hat{n}_1)(\d z) \cr
 \ar\le\ar
\frac{1}{2}Kf(2)\int_1^\infty \{f(\lfloor z\rfloor)+f(1)\} (\hat{m}_1 + \hat{n}_1)(\d z) \cr
 \ar\le\ar
\frac{1}{2}Kf(2)\bigg\{\int_1^\infty f(\lfloor z\rfloor) (\hat{m}_1 + \hat{n}_1)(\d z) + 2f(1)\bigg\}< \infty,
 \eeqnn
which implies $\int_1^\infty f(z) (m+n)(\d z)< \infty$. \end{proof}



\end{document}